\documentclass[a4paper,11pt]{article}
\usepackage[latin1]{inputenc}
\usepackage[english]{babel}
\usepackage{amsmath,amssymb,amsthm}
\usepackage{graphics}
\usepackage{color}
\newcounter{theorem}

\newtheorem{theo}[theorem]{Theorem}%
\newtheorem{prop}[theorem]{Proposition}%
\newtheorem{defi}[theorem]{Definition}%
\newtheorem{lemm}[theorem]{Lemma}%
\newtheorem{coro}[theorem]{Corollary}%

\usepackage{graphicx, epsfig, psfrag}
\usepackage{caption}
\usepackage[all]{xy}
\newsavebox{\smatrix}

\def\id{\mathop{\mathrm{Id}}}
\def\Ker{{\rm Ker}}

\def\dist{{\rm dist}}

\def\diag{{\rm Diag}}

\def\d{\mathrm{d}}
\def\odc{[\![}
\def\fdc{]\!]}
\def\ep{\varepsilon}
\def\ccf{\mathbb{C}}
\def\ddf{\mathbb{D}}

\def\nnf{\mathbb{N}}
\def\ppf{\mathbb{P}}
\def\qqf{\mathbb{Q}}
\def\rrf{\mathbb{R}}
\def\ttf{\mathbb{T}}

\def\zzf{\mathbb{Z}}

\def\ac{\mathcal{A}}
\def\bc{\mathcal{B}}

\def\lc{\mathcal{L}}
\def\mc{\mathcal{M}}
\def\tc{\mathcal{T}}

\def\xc{\mathcal{X}}

\begin{document}
\title{Linear toral endomorphisms, exactness, Bernoulliness and generators} 
\author{Christophe Leuridan}
\maketitle
\begin{abstract}
  Let $T$ be a non-invertible surjective continuous endomorphism of the group
  $\ttf^d := \rrf^d/\zzf^d$. Then $T$ preserves the Haar measure,
  is associated to some $d \times d$ matrix $A$ with integer
  coefficients and is $r$-to-one, where $r = |\det A| \ge 2$.

  In~\cite{Krzyzewski}, Krzy\.{z}ewski gives a necessary and
  sufficient condition for $T$ to be exact, and shows that $T$ -- viewed as
  a measure-preserving map -- is isomorphic to the direct product of some
  (possibly trivial) toral automorphism by some exact toral endomorphism.
  We re-prove these results by more elementary and constructive methods,
  whereas Krzy\.{z}ewski actually works with endomorphisms of compact
  Abelian groups and relies on non-constructive theorems.

  Krzy\.{z}ewski's exactness condition and Mihailescu's papers
  ~\cite{Mihailescu 2012} and~\cite{Mihailescu 2013} show that $T$ is
  isomorphic to the one-sided uniform Bernoulli shift on
  $\odc 0,r-1 \fdc^{\zzf_+}$ if and only if $A$ is expanding
  (i.e. the modulus of each eigenvalue is $>1$).

  When $A$ is not expanding and hyperbolic (i.e. all eigenvalues have modulus
  different from $1$), Mihailescu~\cite{Mihailescu 2012} shows that $T$ cannot
  have a generating Rokhlin partition. More generally, when $A$ is not
  expanding (whether hyperbolic or not), we show that the endomorphism $T$
  cannot have a smooth finite generator (whether independent or not).
\end{abstract}

\begin{flushleft}
\end{flushleft}

\section*{Introduction}

\subsection*{Linear toral endomorphisms}

Let $T$ be a surjective continuous endomorphism of the
torus $\ttf^d := \rrf^d/\zzf^d$. There exists a unique $d \times d$
matrix with integral coefficients $A \in \mc_d(\zzf)$ such that $T$ is
the map $T_A : x \mapsto Ax$ from $\ttf^d$ to itself (note that
the product $Ax$ makes sense since $A$ has integral coefficients).
Let $r := |\det A|$. Then the endomorphism $T$ preserves the Haar
measure $\eta$ on $(\ttf^d,\bc(\ttf^d))$ and is $r$-to-one: for
every $y \in \ttf^d$, the preimage $T^{-1}(\{y\})$ has exactly $r$ elements.

Furthermore, given a uniform random variable $X_0$ taking values in $\ttf^d$
and with law $\eta$, the distribution of $T(X_0)$ is still $\eta$, and for
$\eta$-almost every $y \in \ttf^d$, the conditional law $\lc(X_0|T_A(X_0)=y)$
is uniform on the set $T^{-1}(\{y\})$.
In particular, $T$ is an automorphism of the probability space
$(\ttf^d,\bc(\ttf^d),\eta)$ if and only if $r=1$.

The linear toral automorphisms have been extensively studied, whereas the
linear toral surjective endomorphisms have been studied much less.
From now on, we focus on the ergodic properties of non bijective
endomorphisms and assume that $r := |\det A| \ge 2$. 

A rather simple
application of Fourier series shows that $T$ is ergodic if and only if none
of the $A$'s eigenvalues are roots of unity. 
In 1971, Katznelson~\cite{Katznelson} proved that all
ergodic linear toral automorphisms are Bernoulli, and that the
\emph{natural extension} of any ergodic linear toral endomorphism
is still Bernoulli. Yet, this weaker property does not ensure that the
endomorphism itself is one-sided Bernoulli.

\subsection*{Exactness}

To be one-sided Bernoulli, the endomorphism $T$ must be exact, i.e.
the tail $\sigma$-field $\bigcap_{n \ge 0}T^{-n}\bc(\ttf^d)$ must
be trivial with regard to the Haar measure $\eta$.
For example, when $A$ is a block-diagonal matrix $A = \diag(A_1,A_2)$,
the endomorphism $T$ is isomorphic to the direct product
$T_{A_1} \times T_{A_2}$. If $A_1$, say, has determinant $\pm 1$
(and size $\ge 1$), the endomorphism $T$ cannot be exact since
$T_{A_1}$ is invertible.

In 1993, Krzy\.{z}ewski~\cite{Krzyzewski} provided a necessary and sufficient
condition for $T$ to be exact, given by the decomposition of the characteristic
polynomial $\chi_A$ into irreducible factors (over $\zzf$ or equivalently
over $\qqf$, since $\chi_A$ is monic): the endomorphism $T$ is exact if
and only if the only unimodular (i.e. monic, with integral coefficients and
constant term $\pm 1$) factor of $\chi_A$ is $1$.
Therefore, if $A$ has an eigenvalue with modulus $1$, this eigenvalue
is an algebraic integer, its minimal polynomial divides $\chi_A$ and is
unimodular
(see Lemma~\ref{minimal polynomial of an unimodular algebraic integer}).
Thus $T$ is not exact, so it cannot be one-sided Bernoulli.

Moreover, with the help of Lind Splitting Theorem~\cite{Lind},
Krzy\.{z}ewski also proved that $T$ viewed as a measure-preserving map is
isomorphic to the direct product of some (possibly trivial) toral
\emph{automorphism} and some \emph{exact} surjective endomorphism.
Yet, this isomorphism of dynamical systems is not necessarily an
algebraic isomorphism (i.e. a continuous group isomorphism), and 
algebraic isomorphisms may fail to exist.

Actually, Krzy\.{z}ewski considers also endomorphisms of compact
Abelian groups, and his proof relies heavily on the duality of locally
compact Abelian groups. In the present paper, we focus only on toral
endomorphisms and we give self-contained and more elementary proofs,
analyzing the structure of some subgroups on $\rrf^d$.

\subsection*{Bernoulliness}

Determining whether the endomorphism $T$ is one-sided Bernoulli was a more
difficult problem. The case where $A$ has an eigenvalue with modulus $1$
was settled by the consequence of Krzy\.{z}ewski's exactness condition
mentioned above:
in this case, $T$ is not exact, so it cannot be one-sided Bernoulli.
It therefore remained to study the hyperbolic case (i.e. the case where
$A$ has no eigenvalue with modulus $1$). This work was carried out by
Mihailescu in 2012 and 2013.

On the one hand, in~\cite{Mihailescu 2012}, Mihailescu studies ergodic
properties of equilibrium measures of hyperbolic non-invertible maps
on basic sets with overlaps.
Corollary 1 of~\cite{Mihailescu 2012} states that if $A$ is hyperbolic
but not expanding, then $T$ is not one-sided Bernoulli. 
On the other hand, in~\cite{Mihailescu 2013}, Mihailescu shows that
expanding endomorphisms
on arbitrary tori are one-sided Bernoulli with respect to their corresponding
measure of maximal entropy. In particular, if $A$ is expanding,
$T$ is one-sided Bernoulli.

As a result, $T$ is one-sided Bernoulli if and only if $A$ is expanding. 

\subsection*{Partitions and generators}

To each finite partition $\pi = \{P_1,\ldots,P_s\}$ of $\ttf^d$ into
Borel subsets, we associate the map $\pi : \ttf^d \to \odc 1,s \fdc$
defined by $\pi(x) := i$ whenever $x \in P_i$ and the $\pi$-name map
$\Psi_\pi : \ttf^d \to \odc 1,s \fdc^{\nnf_0}$ defined by $\Psi_\pi(x) :=
\big(\pi(T^n(x))\big)_{n \ge 0}$. 

Many properties of $\pi$ are worth investigating. 
\begin{itemize}
\item The partition $\pi$ is a \emph{(one-sided) generator} of $T$ if the
  partitions $(T^{-n}\pi)_{n \ge 0}$ generate $\bc(\ttf^d)$
  (modulo the null sets), or equivalently, separate the points of
  (a full measure subset of) $\ttf^d$.
\item The translated partitions $(T^{-n}\pi)_{n \ge 0}$ may be independent
  for the probability measure $\eta$.
\item When both conditions above hold, one says that $\pi$ is an
  independent generator.   
\item The partition $\pi$ is a \emph{Rokhlin partition} if for each
  $i \in \odc 1,s \fdc$, the restricted map $T\big|_{P_i}$ is a bijection
  from $P_i$ to $\ttf^d$ (modulo the null sets). 
\item 
  Lind~\cite{Lind} says that $\pi$ is \emph{smooth} when the boundaries
  of $P_1,\ldots,P_s$ are piecewise smooth differentiable submanifolds. 
\end{itemize}
The factor map $\Psi_\pi$ transforms the measure-preserving map $T$ on
$(\ttf^d,\eta)$ into the
one-sided shift operator $S$ on $(\odc 1,s \fdc^{\nnf_0},(\Psi_\pi(\eta))$.
Moreover, 
\begin{itemize}
\item The partition $\pi$ is a generator if and only if 
  $\Psi_\pi$ is an isomorphism of dynamical systems.
\item If the translated partitions $(T^{-n}\pi)_{n \ge 0}$ are independent,
  then $\Psi_\pi(\eta)$ is the product measure $\pi(\eta)^{\otimes \nnf_0}$,
  so $(\odc 1,s \fdc^{\nnf_0},\Psi_\pi(\eta),S)$ is a $1$-sided
  Bernoulli shift.
\item If $\pi$ is an independent generator, then $\pi$ is a Rokhlin partition. 
\item If $\pi$ is a Rokhlin partition, then $T$ is $s$-to-one
  (modulo the null sets), so $s=r$. 
\end{itemize}

\subsection*{Existence or non-existence theorems}

Theorem~1 of~\cite{Mihailescu 2013} states that if $A$ is \emph{expanding}, 
then $T$ is isomorphic to the uniform Bernoulli shift on
$\odc 1,r \fdc^{\nnf_0}$.
Equivalently, $T$ admits an independent generator, which is a partition
into $r$ blocks with measure $1/r$ each. 
Mihailescu's proof is not constructive and relies on a criterion
introduced by Hoffman and Rudolph~\cite{Hoffman - Rudolph}, inspired by
Vershik's standardness criterion \cite{Vershik,Ver95}
and called \emph{tree very weak Bernoulli}.

When $A$ is expanding, an explicit independent generator is known in
many situations. A natural approach to get one is to imitate the
$r$-adic expansion of real numbers. Fix a system $\ddf$ of representatives of
$\zzf^d/A\zzf^d$. Then $|\ddf| = r$. The compact set
$$K_\ddf := \sum_{n \ge 1} A^{-n}\ddf$$
is called a \emph{self-affine tile} since $K_\ddf$ is the almost disjoint
union of the sets $A^{-1}K_\ddf+A^{-1}v$ over all $v \in \ddf$.
Hacon, Saldanah and Veerman note in~\cite{Hacon - Saldanha - Veerman} that the
Lebesgue measure of $K_\ddf$ is always a positive integer.
When it equals $1$, an independent generator can easily be derived
from $K_\ddf$. Most of the time, the blocks of such
a generator have a fractal structure. Unfortunately, it is not always
possible to choose $\ddf$ so that the Lebesgue measure of $K_\ddf$ is $1$,
although it is always possible when $d=2$. Those two facts are proved
in~\cite{Leuridan}, with the help of
a criterion provided by Lagarias and Wang~\cite{Lagarias - Wang}.

By a different approach, Leuridan (Proposition 4 in~\cite{Leuridan})
constructs from any bounded fundamental domain of $\rrf^d/\zzf^d$
a Rokhlin partition $\pi$ such that the translated partitions
$(T^{-n}\pi)_{n \ge 0}$ are independent. Although the paper~\cite{Leuridan}
assumes that $A$ is expanding, this construction still works when $A$
is not expanding.
When the fundamental domain is a parallelepiped, the blocks of $\pi$
are finite unions of polyhedra, so the partition $\pi$ is smooth.  
Under the more restrictive condition $\Vert A^{-1} \Vert < 1$, where
$\Vert\cdot\Vert$ denotes the operator norm associated to the norm on $\rrf^d$
given by the maximum of the absolute values of all coordinates,
Theorem 8 of~\cite{Leuridan} provides an explicit independent
generator in which all blocks are polyhedra.

When $A$ is hyperbolic but not expanding, Mihailescu
(Corollary 2b in \cite{Mihailescu 2012}) provides a negative partial result: 
no Rokhlin partition can be a generator of $T$. 

When $A$ (hyperbolic or not) has an eigenvalue with modulus $<1$, we prove
a negative partial result: the endomorphism $T$ has no smooth
finite generator. More precisely, fix a norm
$|\cdot|$ on $\rrf^d$, and endow $\ttf^d$  with the associated quotient metric. 
For every Borel subset $B$ of $\ttf^d$, and every $\ep>0$, let
$B_\ep := \{x \in \ttf^d : \dist(x,B)<\ep\}$. Call $\eta$ the
Haar measure on $\ttf^d$. 


\begin{theo}\label{no smooth generator}
  Let $\pi = \{P_1,\ldots,P_s\}$ be a finite measurable partition of $\ttf^d$.
  Set $\partial\pi = \partial P_1 \cup \cdots \partial P_s$.
  If $A$ has an eigenvalue with modulus $<1$ and if
  $$\int_0^1 \frac{\eta((\partial\pi)_\ep)}{\ep} \d\ep < +\infty,$$
  the partition $\pi$ cannot be a generator of $T$. 
\end{theo}

By the equivalence of all norms on $\rrf^d$, the condition of
Theorem~\ref{no smooth generator} does not depend on the choice of the norm. 
Since $\eta((\partial\pi)_\ep) \to \eta(\partial\pi)$ as $\ep \to 0$,
this condition forces $\eta(\partial\pi)$ to be $0$, but the convergence of
$\eta((\partial \pi)_\ep)$ to $0$ can be very slow: assuming that  
$\eta((\partial \pi)_\ep) = O(|\ln \ep|^{-\alpha})$ as $\ep \to 0$
for some $\alpha>1$ is sufficient. Therefore, assuming that the upper box
dimension of $\partial \pi$ is $<d$, i.e. 
$\eta((\partial \pi)_\ep) = O(\ep^{\alpha})$ as $\ep \to 0$,
for some $\alpha > 0$, is sufficient. 

Theorem~\ref{no smooth generator} can be compared with a theorem of Lind,
mentioned at page 99 of Bowen's paper~\cite{Bowen}: a linear hyperbolic
\emph{automorphism} of $\ttf^d$ cannot have a \emph{smooth partition}
which is a Bernoulli generator. Bowen used the terminology \emph{Anosov}
instead of \emph{hyperbolic}.

\section{Proof of Theorem~\ref{no smooth generator}}

Since $A$ is not expanding, it has an eigenvalue $\alpha \in \ccf$
such that $|\alpha|<1$. Set $Q = X-\alpha$ if $\alpha$ is real
and $Q = (X-\alpha)(X-\overline{\alpha})$ otherwise. 
In both cases, $Q$ is irreducible in $\rrf[X]$ and the vector subspace
$S := \Ker~Q(A)$ is nonzero and can be written
as a direct sum of cyclic vector subspaces of the linear map
$x \mapsto Ax$. Therefore, we may
- and we do - endow $\rrf^d$ with a norm $|\cdot|$ such that
$|Ax| = |\alpha| \times |x|$ for all $x \in S$. We endow $\ttf^d$  
with the associated quotient metric. 

For all $x \in \ttf^d$, let $R(x) := \dist(x,\partial \pi)$ and   
$\rho(x) := \inf\{|\alpha|^{-n}R(T^n(x)) : n \ge 0 \}$.
Since $\eta(\partial\pi) = 0$, one has $R(x)>0$ for $\eta$-almost
every $x \in \ttf^d$.

We claim that we have also $\rho(x)>0$ for $\eta$-almost every $x \in \ttf^d$. 
Indeed, let $X$ be a $\ttf^d$-valued random variable with distribution
$\eta$, defined on some probability space $(\Omega,\ac,\ppf)$.
For every integer $n \ge 0$, the distribution of $T^n(X)$ is $\eta$,
so
$$\ppf\big[R(T^n(X)) = 0\big] = \eta(\partial\pi) = 0,$$ 
i.e.  $R(T^n(X))>0$ almost surely. Furthermore,
$$\ppf\big[|\alpha|^{-n}R(T^n(X))<1\big]
= \eta\big(\{x \in \ttf^d : R(x) < |\alpha|^n\}\big)
= \eta((\partial \pi)_{|\alpha|^n}).$$
The assumptions of Theorem~\ref{no smooth generator} ensure
the convergence of series $\sum_n \eta((\partial \pi)_{|\alpha|^n})$.
By Borel--Cantelli Lemma, almost surely,
$|\alpha|^{-n}R(T^n(X)) \ge 1$ for every large enough $n$. 
As a result, $\rho(X)>0$ almost surely, which proves the claim.

Last, given $x \in \ttf^d$ such that $\rho(x)>0$, we show that
the partition $\bigvee_{n \ge 0}T^{-n}\pi$ does not distinguish $x$
from points of the form $x+v$ with $v \in S$ such that
$|v|<\rho(x)$. Indeed, given such a vector $v$, we have
for every $n \ge 0$, $T^n(x+v) = T^n(x)+A^nv$ and
$|A^nv| = |\alpha|^n|v| < |\alpha|^n\rho(x) \le R(T^n(x))$,   
so $T^n(x+v)$ and $T^n(x)$ belong to the same block of the
partition $\pi$. We are done.

\section{Simpler proof of Krzy\.zewski's theorems}

Before stating and proving Krzy\.{z}ewski's theorems, we need a definition.

\begin{defi}
  A monic polynomial $f$ with integral coefficients is
  \emph{unimodular} if its constant term is $\pm 1$.
\end{defi}

Note that a product of polynomials with integral coefficients is
unimodular if and only if all factors are unimodular.

\subsection{Statements and steps of the proof}

First, we establish (Proposition~\ref{criterion via the density of Gamma})
a rather simple criterion.
Let $\Gamma = \bigcup_{n \ge 0} A^{-n}\zzf^d$. Then $\Gamma$ is a subgroup 
of $\rrf^d$.
Furthermore, $T$ is exact if and only if $\Gamma$ is dense in $\rrf^d$.

Then, we factor the characteristic polynomial as $\chi_A= f_1f_2$, where
$f_1$ is the product of its unimodular irreducible factors and 
$f_2$ the product of its non-unimodular irreducible factors.
We have $\rrf^d = K_1 \oplus K_2$,
where $K_i := \Ker f_i(A)$ for $i \in \{1,2\}$.
Call $p_1$ the projection onto $K_1$ along $K_2$. We prove that
$\overline{\Gamma} = p_1(\zzf^d) \oplus K_2$
(Proposition~\ref{closure of Gamma}). 
Moreover, the subgroup $p_1(\zzf^d)$ is a lattice of $K_1$. 

From Propositions~\ref{criterion via the density of Gamma} 
and~\ref{closure of Gamma}, we derive Krzy\.{z}ewski's criterion. 

\begin{theo}\label{Krzyzewski's condition for exactness}
  {\bf Krzy\.{z}ewski~\cite{Krzyzewski}.}
  $T$ is exact if and only if the only 
  unimodular factor of the characteristic polynomial $\chi_A$ is $1$. 
\end{theo}

Let $p_{\ttf^d}$ denote the canonical projection from $\rrf^d$ onto
$\ttf^d$ and consider the subgroups $H_i := p_{\ttf^d}(K_i)$ for
$i \in \{1,2\}$. The equality $\rrf^d = K_1 \oplus K_2$ yields
$\ttf^d = H_1 + H_2$. Yet, this latter sum is not necessarily 
direct, and this causes difficulties. 

We prove that $H_2$ is a subtorus of $\ttf^d$ and that $\ttf^d/H_2$
is a torus.
Taking quotients, we prove $T$ yields an \emph{automorphism}
$T_{\ttf^d/H_2}$ of $\ttf^d/H_2$, whereas $T$ induces an \emph{exact}
endomorphism $T_{H_2}$ of $H_2$. See Proposition~\ref{more on Gamma and H}.

Next, we represent $T$ as a skew product of $T_{\ttf^d/H_2}$ and $T_{H_2}$
(Proposition~\ref{skew product}). 
Finally, we establish Theorem~\ref{splitting theorem}
-- which is a variant of a thoerem of Lind --
to remove the cocycle and turn this skew product into the direct product
of the automorphism $T_{\ttf^d/H_2}$ and the exact endomorphism $T_{H_2}$.

Actually, Proposition~\ref{skew product},
Lemma~\ref{Borel choice of representative} and
Theorem~\ref{splitting theorem} show that it suffices to choose
fundamental domains $K_1/p_1(\zzf^d)$ and $K_2/(K_2 \cap \zzf^d)$
to obtain the explicit isomorphisms of dynamical systems
$$T \quad
\simeq \quad T_{\ttf^d/H_2} \times_\alpha T_{H_2} \quad 
\simeq \quad T_{\ttf^d/H_2} \times T_{H_2}.$$

\subsection{First condition for $T$ to be exact}

Call $p_{\ttf^d}$ the canonical projection from $\rrf^d$ to
$\ttf^d = \rrf^d/\zzf^d$. Set
$$\tc := \bigcap_{n \ge 0}T^{-n}\bc(\ttf^d),
\quad \Gamma := \bigcup_{n \ge 0} A^{-n}\zzf^d
\text{ and }
H := \bigcup_{n \ge 0} \Ker~T^n = p_{\ttf^d}(\Gamma).$$
Then $\Gamma$ is a subgroup of $\rrf^d$ as the union of 
an increasing sequence of subgroups of $\rrf^d$, and 
$H$ is a subgroup of $\ttf^d$.

The Borel functions on the torus $\ttf^d$ (with values in $\rrf$, say)
are exactly the functions obtained by quotienting all Borel
$\zzf^d$-periodic functions on $\rrf^d$. Given $n \ge 0$,
one checks that the $T^{-n}\bc(\ttf^d)$-measurable functions on $\ttf^d$
are exactly the $(\Ker~T^n)$-periodic Borel functions on $\ttf^d$,
equivalently the functions obtained by quotienting (by $\zzf^d$) all Borel
$A^{-n}\zzf^d$-periodic functions on $\rrf^d$.

Hence, the $\tc$-measurable functions on $\ttf^d$ are exactly the
$H$-periodic Borel functions on $\ttf^d$, equivalently the functions
obtained by quotienting (by $\zzf^d$) all Borel $\Gamma$-periodic
functions on $\rrf^d$. This observation is the first step to get a
necessary and sufficient condition for $T$ to be exact.

\begin{prop}\label{criterion via the density of Gamma}
  The group $p_{\ttf^d}(\overline{\Gamma})$
  is compact and is the closure of $H = p_{\ttf^d}(\Gamma)$.
  Furthermore, the following statements are equivalent:
  \begin{enumerate}
  \item $T$ is exact;
  \item $H$ is dense in $\ttf^d$;
  \item $\Gamma$ is dense in $\rrf^d$.
  \end{enumerate}
\end{prop}

\begin{proof}
  Since $p_{\ttf^d}$ is a continuous group morphism and since
  $\overline{\Gamma}$ is a closed subgroup containing $\zzf^d$, the set
  $p_{\ttf^d}(\overline{\Gamma}) = p_{\ttf^d}(\overline{\Gamma} \cap [0,1]^d)$
  is compact and is the closure of $H = p_{\ttf^d}(\Gamma)$.
  Hence, $H$ is dense in $\ttf^d$ if and only if $\overline{\Gamma}$
  contains at least one point in each equivalence class modulo $\zzf^d$,
  which is equivalent to $\overline{\Gamma} = \rrf^d$.
  The equivalence of items 2 and 3 follows.

  Let $f : \ttf^d \to \rrf$ be any bounded $\tc$-measurable function.
  Then the map $h \mapsto f(\cdot-h)$ from $\ttf^d$ to $L^1(\eta)$
  is continuous and $H$-periodic. If $H$ is dense in $\ttf^d$,
  this map is constant, so $f$ is almost everywhere constant,
  which shows that the tail $\sigma$-field $\tc$ is trivial.

  Conversely, if $H$ is not dense in $\ttf^d$, the tail $\sigma$-field
  $\tc$ is not trivial: given any norm on $\rrf^d$, the associated
  quotient metric on $\ttf^d$ is invariant by translations, so the
  continuous function $\dist(\cdot,H)$ is $H$-periodic but not constant. 
  The equivalence of items 1 and 2 follows. 
\end{proof}

\subsection{Description of the closure of $\Gamma$}

We consider the factorization of $\chi_A$ into monic irreducible
(in $\zzf[X]$, or equivalently in $\qqf[X]$) polynomials.
Call $f_1$ the product of its unimodular irreducible factors and 
$f_2$ the product of its non-unimodular irreducible factors,
counted with their multiplicities. From this factorisation,
we get a complete description of the closure of $\Gamma$ and other
useful properties.

\begin{prop}\label{closure of Gamma}  
  Let $K_1 := \Ker f_1(A)$ and $K_2 := \Ker f_2(A)$. Then
  \begin{enumerate}
  \item One has $\rrf^d = K_1 \oplus K_2$. The associated projections
    $p_1$ and $p_2$ onto $K_1$ and $K_2$ are polynomials of the
    endomorphism $u : x \mapsto Ax$ with rational coefficients. 
    Thus the matrices of $p_1$ and $p_2$ in the canonical basis
    have rational entries.

  \item The subspaces $K_1$ and $K_2$ are stable by the endomorphism
    $u : x \mapsto Ax$ of $\rrf^d$.
    The characteristic polynomial of the endomorphisms $u_1$ and $u_2$
    induced by
    $u$ on $K_1$ and $K_2$ are $f_1$ and $f_2$. In particular,
    $\det u_1 = \pm 1$. 
  \item One has $\overline{\Gamma} = p_1(\zzf^d) \oplus K_2$ and
    $p_1(\overline{\Gamma}) = p_1(\Gamma) = p_1(\zzf^d)$.
    Moreover, the subgroup $p_1(\zzf^d)$ is a lattice of $K_1$. 
  \item One has $p_{\ttf^d}(\overline{\Gamma}) = p_{\ttf^d}(K_2)$.   
  \end{enumerate}
\end{prop}

\begin{proof}
Proposition~\ref{closure of Gamma} is a direct application of
proposition~\ref{structure of Gamma} to $V = \rrf^d$, $\Lambda = \zzf^d$
and to the endomorphism $u : x \mapsto Ax$ of $\rrf^d$.
See section~\ref{Annex}.
\end{proof}

\subsection{Factorisation of $T$ as a skew product}


Keep the notations of the previous subsections and set
$H_i := p_{\ttf^d}(K_d)$ for $i \in \{1,2\}$. Since $A\Gamma = \Gamma$, we have
$A\overline{\Gamma} = \overline{\Gamma}$. Applying $p_{\ttf^d}$
to both sides yields $T(\overline{H}) = \overline{H} = H_2$, thanks to
item 4 of Proposition~\ref{closure of Gamma}.  
Hence $H_2$ is compact, $T$ induces an endomorphism $T_{H_2}$ of the
compact group $H_2$, and provides by quotienting an endomorphism
$T_{\ttf^d/H_2}$ of $\ttf^d/H_2$.
By Lemma~\ref{image of Haar measure by a surjective morphism}, the
surjective continuous group endomorphisms $T_{H_2}$ and $T_{\ttf^d/H_2}$
preserve the Haar measures $\eta_{H_2}$ and $\eta_{\ttf^d/H_2}$.

\begin{prop}\label{more on Gamma and H}
  Keep the notations of Proposition~\ref{closure of Gamma}, and call
  $d_1 = \dim K_1 \in \odc 0,d-1 \fdc$ and $d_2 = \dim K_2 = d-d_1$.
  Then
  \begin{enumerate}
  \item The compact group $\ttf^d/H_2$ is canonically isomorphic
    to $\rrf^d/\overline{\Gamma}$ and to $K_1/p_1(\zzf^d)$, and
    is isomorphic to $\ttf^{d_1}$.
  \item The endomorphism $T_{\ttf^d/H_2}$ is an automorphism.
  \item The surjective endomorphism $T_{H_2}$ is exact
    (with regard to $\eta_{H_2}$).
  \item The compact group $H_2$ is canonically isomorphic to
    $K_2/(K_2 \cap \zzf^d)$ and is isomorphic to $\ttf^{d_2}$. 
  \end{enumerate}
\end{prop}

\begin{proof}
  Remind that $H_2 = p_{\ttf^d}(\overline{\Gamma})$ and that
  $\overline{\Gamma} \supset \zzf^d$. 
  The canonical endomorphism $x \mapsto p_{\ttf^d}(x)+H_2$ from
  $\rrf^d$ to $\ttf^d/H_2$ is surjective and
  its kernel is $\overline{\Gamma}$ since for all $x \in \rrf^d$, 
  \begin{eqnarray*}
    p_{\ttf^d}(x)+H_2 = H_2
    \iff p_{\ttf^d}(x) \in p_{\ttf^d}(\overline{\Gamma})
    \iff x \in \overline{\Gamma}+\zzf^d
    \iff x \in \overline{\Gamma}.
  \end{eqnarray*}
  Taking quotients yields a canonical isomorphism from
  $\ttf^d/H_2$ to $\rrf^d/\overline{\Gamma}$.
  The canonical isomorphism with $K_1/p_1(\zzf^d)$ follows from the
  equalities $\rrf^d = K_1 \oplus K_2$ and $\rrf^d = p_1(\zzf^d) \oplus K_2$
  and the inclusion $p_1(\zzf^d) \subset K_1$. And since $p_1(\zzf^d)$ is a
  lattice of $K_1$, the quotient $K_1/p_1(\zzf^d)$
  is isomorphic to $\ttf^{d_1}$. Item~1 follows.


  The endomorphism $T$ of $\ttf^d$ is surjective. Taking quotients,
  we derive that the endomorphism $T_{\ttf^d/H_2}$ of $\ttf^d/H_2$
  is surjective, so we have only to check that its kernel contains
  only $0_{\ttf^d/H_2} = H_2$. Let $C = x+H_2$ be a coset in
  $\Ker T_{\ttf^d/H_2}$. Since $\ttf^d = H_1+H_2$, one has $x = x_1+x_2$ 
  for some $(x_1,x_2) \in H_1 \times H_2$, so $C = x_1+H_2$.  
  Hence, we have
  $H_2 = 0_{\ttf^d/H_2} = T_{\ttf^d/H_2}(C) = Ax_1 + H_2$, i.e. $Ax_1 \in H_2$,
  so $A^kx_1 \in H_2$ for every $k \ge 1$.
  
  Since $H_1 = p_{\ttf^d}(K_1)$, one has $x_1 = p_{\ttf^d}(v_1)$ for some
  $v_1 \in K_1$. By definition of $K_1$, one has $f_1(A)v_1 = 0_{\rrf^d}$.  
  Quotienting by $\zzf^d$, we derive $f_1(A)x_1 = 0_{\ttf^d}$. Since 
  $f_1$ is unimodular, we get that $x_1$ is a linear combination with
  integral coefficients of $Ax_1,\ldots,A^dx_1$. Thus $x_1 \in H_2$
  and $C = H_2 = 0_{\ttf^d/H_2}$. Item~2 follows. 
  
  Since $T(H_2) = H_2$, the endomorphism $T$ induces a surjective
  endomorphism $T_{H_2}$ on $H_2$, which
  necessarily preserves the Haar measure $\eta_{H_2}$ on $H_2$
  (see Lemma~\ref{image of Haar measure by a surjective morphism}).
  Let $f : H_2 \to \rrf$ be a bounded $\tc_{2}$-measurable function, where
  $\tc_{2} := \bigcap_{n \ge 0}T_{H_2}^{-n}\bc(H_2)$.
  Then $f$ is $H$-periodic, therefore almost everywhere constant,
  by continuity of the map $h \mapsto f(\cdot-h)$ from $H_2$
  to $L^1(\eta_{H_2})$. Item~3 follows.


  
  Since $K_2$ can be defined by linear equations with integral coefficients,
  $K_2 \cap \zzf^d$ is a lattice of $K_2$.
  We derive the isomorphisms like in item~1. Item~4 follows.  
\end{proof}



We now apply Lemma~\ref{quotient by a closed normal subgroup}
given in the annex.
For every coset $C=a+H_2$, the probability measure $\delta_a* \eta_{H_2}$
(i.e. the image of $\eta_{H_2}$ by the translation $y \mapsto a+y$ in $\ttf^d$)
is carried by $C$ and depends only on $C$.
We call it {\it uniform distribution on} $C$ and denote it by $\eta_C$
(this notation is coherent when $C=H_2$).
Then the disintegration of $\eta$ by the canonical projection 
$p_{\ttf^d/H_2} : \ttf^d \to \ttf^d/H_2$ is given by

$$\eta = \int_{\ttf^d/H_2} \eta_C ~ \d\eta_{\ttf^d/H_2}(C).$$ 

These observations will help us to prove that
the endomorphism $T$ of the probability space $(\ttf^d,\eta)$ is
isomorphic to a skew product of $T_{\ttf^d/H_2}$ and $T_{H_2}$. 

\begin{prop}\label{skew product}
  Let $f : \ttf^d/H_2 \to \ttf^d$ be any Borel choice of
  representatives, i.e. any Borel map such that
  $p_{\ttf^d/H_2} \circ f = \id_{\ttf^d/H_2}$.
  For every $C \in \ttf^d/H_2$, set
  $$\alpha(C) := T(f(C)) - f\big(T_{\ttf^d/H_2}(C)\big).$$
  \begin{enumerate}
  \item This formula defines a Borel map $\alpha$ from $\ttf^d/H_2$ to $H_2$.
  \item The skew-product map $T_{\ttf^d/H_2} \times_\alpha T_{H_2} :
    \ttf^d/H_2 \times H_2 \to \ttf^d/H_2 \times H_2$ given by
    $$\big(T_{\ttf^d/H_2} \times_\alpha T_{H_2} \big) (C,y)
    := (T_{\ttf^d/H_2}(C),T_{H_2}(y)+\alpha(C))$$
    is measurable and preserves the measure
    $\eta_{\ttf^d/H_2} \otimes \eta_{H_2}$.  
  \item The map $\Phi : \ttf^d/H_2 \times H_2 \to \ttf^d$
    defined by $\Phi(C,y) := f(C)+y$ is an isomorphism transforming
    $(\ttf^d/H_2 \times H_2~,~\eta_{\ttf^d/H_2} \otimes \eta_{H_2}~,
    ~T_{\ttf^d/H_2} \times_\alpha T_{H_2})$ into $(\ttf^d,\eta,T)$. 
  \end{enumerate}
\end{prop}

\begin{proof}
  For every coset $C \in \ttf^d/H_2$, the elements $T(f(C))$
  and $f\big(T_{\ttf^d/H_2}(C)\big)$ belong to the same coset
  $T_{\ttf^d/H_2}(C)$, so their difference $\alpha(C)$
  belongs to $H_2$. Item 1 follows.

  One checks item 2 by applying Fubini's theorem.

  The map $\Phi$ is measurable since $f$ is. Set
  $\pi := \Phi\big(\eta_{\ttf^d/H_2} \otimes \eta_{H_2}\big) =
  f_*(\eta_{\ttf^d/H_2})*\eta_{H_2}$.
  Since $\eta_{H_2} * \eta_{H_2} = \eta_{H_2}$, we have $\pi * \eta_{H_2} = \pi$. 
  For every $C \in \ttf^d/H_2$ and $y \in H_2$, we have
  $p_{\ttf^d/H_2} \big(\Phi(C,y)) = f(C) + y + H_2 = C$.
  Thus $p_{\ttf^d/H_2}(\pi) = \eta_{\ttf^d/H_2}$ and the
  disintegration of $\eta$ by $p_{\ttf^d/H_2}$ is given by
  $$\pi = \int_{\ttf^d/H_2} \pi_C ~ \d\eta_{\ttf^d/H_2}(C),$$
  where $\pi_C$ is some probability measure carried by $C$, depending
  measurably on $C$. Hence
  $$\pi = \pi*\eta_{H_2}
  = \int_{\ttf^d/H_2} (\pi_C*\eta_{H_2}) ~ \d\eta_{\ttf^d/H_2}(C)
  = \int_{\ttf^d/H_2} \eta_C ~ \d\eta_{\ttf^d/H_2}(C) = \eta,$$
  by Lemma~\ref{quotient by a closed normal subgroup} given in the annex.
  
  
  Let $x \in \ttf^d$. For every $C \in \ttf^d/H_2$ and $y \in H_2$,
  \begin{eqnarray*}
    \Phi(C,y) = x
    &\iff& \big( f(C)+y = x \text{ and } C = x+H_2 \big) \\
    &\iff& \big( C = x+H_2 \text{ and } y = x-f(x+H_2) \big).
  \end{eqnarray*}  
  Since $x-f(x+H_2) \in H_2$, we derive that $\Phi$ is bijective and
  $\Phi^{-1}$ is measurable.
  
  Last, for every $C \in \ttf^d/H_2$ and $y \in H_2$,
  \begin{eqnarray*}
  \Phi\big(T_{\ttf^d/H_2} \times_\alpha T_{H_2} \big)
  (C,y)
  &=& f\big(T_{\ttf^d/H_2}(C)\big)
      + T_{H_2}(y) + \alpha(C) \\
  &=& T(f(C)) + T_{H_2}(y) \\  
  &=& T(f(C)+y) = T(\Phi(C,y))
  \end{eqnarray*}
  Item 3 follows. 
\end{proof}

A similar statement is given by Krzy\.{z}ewski in Proposition 2
of \cite{Krzyzewski} in a more general context (namely for an automorphism
of an arbitrary second-countable compact Abelian group).
The proof given there is much more involved since it relies on Abramov's
Lemmas 5 and 6 in~\cite{Abramov} (which uses a general theorem of Rokhlin
result to get a measurable choice of representatives) and uses Fourier
analysis on compact Abelian groups to prove that the skew-product map is
a metric isomorphism of dynamical systems.

Our approach is much more basic and is constructive. Indeed, a Borel choice
of representatives $f : \ttf^d/H_2 \to \ttf^d$ can be obtained simply
as follows. 
Fix a fundamental domain $D_1$ of $K_1/p_1(\zzf^d)$, namely a Borel subset
of $K_1$ which contains exactly one representative of each equivalence class
of $K_1/p_1(\zzf^d)$. Then $K_1 = D_1 \oplus p_1(\zzf^d)$, where the sign $\oplus$  
in the Minkovski sum means that each element of $K_1$ can be uniquely written
as the sum of a vector of $D_1$ and a vector of $p_1(\zzf^d)$. Hence 
$$\rrf^d = (D_1 \oplus p_1(\zzf^d)) \oplus K_2 =
D_1 \oplus(p_1(\zzf^d) \oplus K_2) = D_1 \oplus \overline{\Gamma}.$$
The next lemma provides a Borel map $f : \ttf^d/H_2 \to \ttf^d$
such that $p_{\ttf^d/H_2} \circ f = \id_{\ttf^d/H_2}$.

\begin{lemm}\label{Borel choice of representative}
  The sum $\ttf^d = p_{\ttf^d}(D_1) + p_{\ttf^d}(\overline{\Gamma})
  = p_{\ttf^d}(D_1) + H_2$ is direct. Thus every coset
  $C \in \ttf^d/H_2$ has a unique representative $f(C)$ in
  $p_{\ttf^d}(D_1)$ and the map  $f : \ttf^d/H_2 \to \ttf^d$
  thus defined is Borel.
\end{lemm}

\begin{proof}
  The equality $\ttf^d = p_{\ttf^d}(D_1) + p_{\ttf^d}(\overline{\Gamma})
  = p_{\ttf^d}(D_1) + H_2$ follows from the equality
  $\rrf^d = D_1 \oplus \overline{\Gamma}$. Therefore, each 
  coset in $\ttf^d/H_2$ has at least one representative in $D_1$.
  
  Call $p_{D_1}$ the projection from $\rrf^d$ onto $D_1$ along
  $\overline{\Gamma}$. Since $\overline{\Gamma} \supset \zzf^d$,
  the map $p_{D_1}$ is $\zzf^d$-periodic, so there exists a unique
  map $\tilde{p}_{D_1} : \ttf^d \to D_1$ such that
  $p_{D_1} = \tilde{p}_{D_1} \circ p_{\ttf^d}$.
  Since $p_{\ttf^d}(\overline{\Gamma}) \supset H_2$, the map 
  $\tilde{p}_{D_1}$ is $H_2$-periodic, so there exists a unique
  map $\widehat{p}_{D_1} : \ttf^d/H_2 \to D_1$ such that
  $\tilde{p}_{D_1} = \widehat{p}_{D_1} \circ p_{\ttf^d/H_2}$.

  Let $C = p_{\ttf^d/H_2}(x) \in \ttf^d/H_2$.
  Then $x = p_{\ttf^d}(a)$ for some $a \in \rrf^d$. 
  Assume that $y$ is a representative of $C$ in $p_{\ttf^d}(D_1)$.
  Then $y = p_{\ttf^d}(b)$ for some $b \in D_1$ and $x-y \in H_2$. 
  Since $p_{\ttf^d}(a-b) \in H_2 = p_{\ttf^d}(\overline{\Gamma})$,
  we derive $a-b \in \overline{\Gamma} + \zzf^d = \overline{\Gamma}$,
  so $b = p_{D_1}(a) = \tilde{p}_{D_1}(x) = \widehat{p}_{D_1}(C)$ and
  $y = p_{\ttf^d}(\widehat{p}_{D_1}(C))$. 

  As a result, the only representative of $C$ in $D_1$ is
  $f(C) = p_{\ttf^d}(\widehat{p}_{D_1}(C))$. The map
  $f : \ttf^d/H_2 \to \ttf^d$ thus defined is Borel
  as a composition of Borel maps.
\end{proof}

\subsection{Metric isomorphisms from a skew product
  to a direct product}

We have seen that the measure preserving map $T$ is (metrically)
isomorphic to a skew product
$T_{\ttf^d/H_2} \times_\alpha T_{H_2}$, where $\alpha$
is a suitable Borel map from $\ttf^d/H_2$ to $H_2$.
Actually, a general theorem of Lind~\cite{Lind} ensures that this skew product
`algebraically splits'. We now state and prove a variant of Lind's theorem
which is sufficient for our purpose. 

\begin{theo}\label{splitting theorem}
  Let $V$ be a real vector space with finite dimension, $\Lambda$ be a lattice
  of $V$. Fix $s \in GL(V)$, and call $S$ the surjective endomorphism of the
  torus $V/\Lambda$ obtained from $s$ by quotienting.
  Assume that $s$ has no unimodular eigenvalue. 
  
  Given any automorphism $U$ of a probability space $(X,\xc,\mu)$
  and any measurable map $\alpha : X \to V/\Lambda$,
  one can construct explicitly a measurable map $\beta : X \to V/\Lambda$ 
  such that the map $\Psi_\beta : X \times V/\Lambda \to X \times V/\Lambda$
  defined by $\Psi_\beta(x,y) := (x,y+\beta(x))$
   is an isomorphism transforming the skew product
  $U \times_\alpha S$ into the direct product $U \times S$. 
\end{theo}

Let us explain the main differences with Lind's splitting theorem.
\begin{itemize}
\item The first factor can be any automorphism
  (ergodic or not) of any probability space whereas
  Linds takes an ergodic automorphism of a Lebesgue space.  
\item In Lind's theorem, the second factor can be any continuous
  \emph{automorphism} of any Abelian compact group.  
  Actually, Lind's theorem still applies to a continuous surjective
  endomorphism $S$, since the conclusion for the natural extension of $S$
  easily yields the conclusion for $S$. Hence Lind's theorem applies to
  a broader group of second factors than our theorem. 
\item The proof of Lind's theorem is long, involved, and non-constructive,
  whereas the proof in our context is much simpler and constructive. 
\end{itemize}

We now prove Theorem~\ref{splitting theorem}.

\begin{proof}
  For any  Borel map $\beta : X \to V/\Lambda$, the map $\Psi_\beta$
  is bimeasurable and preserves the measure $\mu \otimes \eta_{V/\Lambda}$.
  So the only thing to check is the equality
  $\Psi_\beta \circ (U \times_\alpha S) = (U \times S) \circ \Psi_\beta$
  $\mu \otimes \eta_{V/\Lambda}$-almost everywhere, which is equivalent to
  \begin{equation}\label{condition on beta}
    \alpha(x) + \beta(U(x)) = S(\beta(x))
    \text{ for } \mu \text{-almost every } x \in X.
  \end{equation}
  To construct an explicit solution of the functional
  equation~(\ref{condition on beta}) which works for all $x \in X$,
  we fix a bounded fundamental domain $D$ of $V/\Lambda$.
  For every $x \in X$, call $a(x)$ the only representative in $D$
  of $\alpha(x)$.
  The map $a$ thus defined is Borel and we have $\alpha = p_{V/\Lambda} \circ a$.
  
  We split $s$ into a stable part and an unstable part.
  By assumption, $\chi_s$ has no unimodular root, so it can be factored 
  as $\chi_s = f_{\rm st}f_{\rm un}$ in $\rrf[X]$, where the roots of $f_{\rm st}$
  (respectively $f_{\rm un}$) have modulus $<1$ (respectively $>1$).
  Therefore $V = K_{\rm st} \oplus K_{\rm un}$, where $K_i = \Ker f_i(s)$ for
  $i \in \{\mathrm{st},\mathrm{un}\}$. Call $p_i$ the associated
  projection from $V$ to $K_i$ and $s_i$ the endomorphism induced by
  $s$ onto $K_i$. One can endow $V$ with a norm $|\cdot|$ such that the
  associated operator norms of $s_{\rm st}$ and $s_{\rm un}^{-1}$ satisfy
  $\Vert s_{\rm st} \Vert<1$ and $\Vert s_{\rm un}^{-1} \Vert<1$.
  Let $M := \sup\{|x| : x \in p_{\rm st}(D) \cup p_{\rm un}(D)\}$.

  Then for every $x \in X$ and $n \ge 0$,
  $$\big|s^{n}(p_{\rm st}(a(U^{-n-1}(x)))\big|
  \le \Vert s_{\rm st} \Vert^{n} \times M \text{ and }
  \big|s^{-n-1}(p_{\rm un}(a(U^{n}(x))))\big|
  \le \Vert s_{\rm un}^{-1} \Vert^{n+1} \times M,$$
  so the vector 
  $$b(x) := -\sum_{n \ge 0} s^np_{\rm st}(a(U^{-n-1}(x)))
  + \sum_{n \ge 0} s^{-n-1}p_{\rm un}(a(U^n(x))) \in V$$
  is well defined. Moreover,
  \begin{eqnarray*}
  a(x) + b(U(x))
    &=& p_{\rm st}(a(x)) - \sum_{n \ge 0} s^n(p_{\rm st}(a(U^{-n}(x)))) \\
    & & + p_{\rm un}(a(x)) + \sum_{n \ge 0} s^{-n-1}(p_{\rm un}(a(U^{n+1}(x)))) \\
    &=& - \sum_{n \ge 0} s^{n+1}(p_{\rm st}(a(U^{-n-1}(x))))
        + \sum_{n \ge 0} s^{-n}(p_{\rm un}(a(U^{n}(x)))) \\
    &=& s(b(x)).
  \end{eqnarray*}
  As a result, the map $\beta := p_{V/\Lambda} \circ b$ is a solution of
  equation~\ref{condition on beta} everywhere on $X$.
  The proof is complete.
\end{proof}

It remains to apply Theorem~\ref{splitting theorem} to
$X = \ttf^d/H_2$, $U = T_{\ttf^d/H_2}$, $V = K_2$, $s=u_2$, 
$\Lambda = K_2 \cap \zzf^d$, so $V/\Lambda = H_2$ and $S = T_{H_2}$.
This is possible because the only unimodular factor of $\chi_{u_2} = f_2$
is $1$, so $u_2$ has non eigenvalue of modulus $1$, thanks to the next lemma.

\begin{lemm}\label{minimal polynomial of an unimodular algebraic integer}
  Let $\zeta \in \ccf$ be an algebraic integer (i.e. a root of some monic
  polynomial with integral coefficients). If $|\zeta|=1$ and
  $\zeta \notin \{-1,1\}$, then the minimal polynomial of $\zeta$ is
  self-reciprocal, hence unimodular (and its degree is even).    
\end{lemm}

\begin{proof}
  Call $f$ the minimal polynomial of $\zeta$, and $d$ its degree.
  Then $d \ge 2$.  
  Since $f$ is irreducible, its constant term is not $0$, so the degree
  of its reciprocal polynomial $g := X^df(1/X)$ is also $d$.
  Assume that $|\zeta|=1$. Then $\overline{\zeta} = \zeta^{-1}$ is a root
  of $g$, and $\zeta$ is also a root of $g$, since $g$ have real coefficients. 
  Hence $f$ divides $g$, and since their degrees are the same, $g=cf$
  for some $c \in \qqf^*$. Looking at the constant terms, we derive
  $1=c^2$. The equality $g=-f$ would imply $f(1)=0$, which contradicts
  the irreducibility of $f$. Hence $g=f$. (Moreover, the degree of $f$
  cannot be odd since it would imply $f(-1)=0$, which contradicts
  the irreducibility of $f$.) 
\end{proof}

Combining Proposition~\ref{more on Gamma and H},
Proposition~\ref{skew product} and Theorem~\ref{splitting theorem}
applied to $X = \ttf^d/H_2$ (endowed with its Borel $\sigma$-field),
$U = T_{\ttf^d/H_2}$, $V = K_2$, $s=u_2$, 
$\Lambda = K_2 \cap \zzf^d$ and $S = T_{H_2}$, we get 
a constructive simpler alternative proof of the next theorem.

\begin{theo} {\bf Krzy\.{z}ewski~\cite{Krzyzewski}}
  Given a Borel map $\alpha$ from $\ttf^d/H_2$ to $H_2$, there exists
  a Borel map $\beta$ from $\ttf^d/H_2$ to $H_2$ such that the map
  $\Psi_\beta : \ttf^d/H_2 \times H_2
  \to \ttf^d/H_2 \times H_2$ defined by $\Psi_\beta(C,y) := (C,y+\beta(C))$
  is an isomorphism transforming the skew product
  $T_{\ttf^d/H_2} \times_\alpha T_{H_2}$
  into the direct product $T_{\ttf^d/H_2} \times T_{H_2}$. 
\end{theo}

\subsection{An example}

Call $(e_1,e_2,e_3)$ the canonical basis of $\rrf^3$. Let
$$A := \left(\begin{array}{ccc}
              1 & 1 & 0 \\
              2 & 3 & 1 \\
              0 & 0 & 3
            \end{array} \right).$$
Then $f_1 = X^2-4X+1$ and $f_2 = X-3$.
The subspace $K_1$ is the plane with equation $x_3=0$ whereas
the subspace $K_2$ is the line $\rrf v$ where $v := e_1+2e_2-2e_3$.
Therefore, the projection onto $K_1$ along $K_2$ is given by
$$p_1(x) = (x_1 + x_3/2 ~~~ x_2 + x_3 ~~~ 0)^\top \text{ for every }
x = (x_1 ~~~ x_2 ~~~ x_3)^\top \in \rrf^3.$$
Since $p_1(\zzf^3) = ((1/2)\zzf) \times \zzf \times \{0\}$,
the set $D_1 := [0,1/2[~\times [0,1[~\times \{0\}$ is a fundamental
domain of $K_1/p_1(\zzf^3)$.

For every $x \in \rrf^3$, the only representative of the coset
$C = p_{\ttf^3}(x) + H_2$ in $p_{\ttf^3}(D_1)$ is
$$f(C) = p_{\ttf^3}(\left(\begin{array}{c}
              (2x_1+x_3 - \lfloor 2x_1+x_3 \rfloor)/2 \\
              x_2+x_3  - \lfloor x_2+x_3 \rfloor \\
              0 
                               \end{array} \right))
        = p_{\ttf^3}(\left(\begin{array}{c}
              (2x_1+x_3 - \lfloor 2x_1+x_3 \rfloor)/2 \\
              x_2+x_3 \\
              0 
                           \end{array} \right)).$$
A computation yields                         
$$\alpha(C) = T(f(C)) - f(T_{\ttf^3/H_2}(C)) = p_{\ttf^3}(g(x)e_1),$$
where
$$g(x)
= \big(\lfloor 2x_1+2x_2+3x_3 \rfloor - \lfloor 2x_1+x_3 \rfloor \big)/2.$$
Since $2g(x) \in \zzf$, we have also $\alpha(C) = p_{\ttf^3}\big(g(x)v\big)$.

We follow step by step the proof of Theorem~\ref{splitting theorem}.
The set $D_2 := [0,1[v$ is a bounded fundamental domain of
$K_2/(K_2 \cap \zzf^d) = (\rrf v)/(\zzf v)$ and the only representative 
of $\alpha(C)$ in $D_2$ is $a(C) = \{ g(x) \} v$, where $\{\cdot\}$
denotes the fractional part. 
Since the endomorphisms $x \mapsto Ax$ and $x \mapsto 3x$ coincide on
$K_2$, the corresponding function $b$ is given by 
$$b(C) := \sum_{n \ge 0} \frac{a(T_{\ttf^3/H_2}^n(C))}{3^{n+1}} = h(x)v,
\text{ where }
h(x) := \sum_{n \ge 0} \frac{\big\{g(T^n(x))\big\}}{3^{n+1}}.$$
\subsection{Algebraic endomorphisms of dynamical systems}

In this subsection, we focus on endomorphisms of dynamical systems
which are also epimorphisms, i.e. surjective (continuous) endomorphisms
of compact groups.
Keep the notations of Proposition~\ref{closure of Gamma}.
For $i \in \{1,2\}$, set $H_i := p_{\ttf^d}(K_i)$, so $H_2 = \overline{H}$. 
The epimorphism $T$ of the compact group $\ttf^d$ induces
an epimorphism $T_{H_i}$ of the compact subgroup $H_i$,
which therefore preserves its Haar measure $\eta_{H_i}$. 
Since $K_1 \oplus K_2 = \rrf^d$, we have $H_1 + H_2 = \ttf^d$.
Yet, this last sum is not necessarily direct.  

\begin{prop}\label{finite-to-one factor maps}
  Call $P_1$, $P_2$ the matrices in the canonical basis of the projections
  $p_1$, $p_2$. Fix an integer $q \ge 1$ such that $qP_1$ and $qP_2$ have
  integral entries. Then 
  \begin{enumerate}
  \item The set $L := (K_1 \cap \zzf^d) \oplus (K_2 \cap \zzf^d)
    = \{x \in \rrf^d : (P_1x,P_2x) \in \zzf^d \times \zzf^d\}$ is a subgroup
    of $\zzf^d$ and contains $q\zzf^d$, so $[\zzf^d :  L] = m$ and
    $[L : q\zzf^d] = q^d/m$ for some divisor $m$ of $q^d$. 
    \item The endomorphism $\Sigma : (x_1,x_2) \mapsto x_1+x_2$ from
      $H_1 \times H_2$ to $\ttf^d$ is a $m$-to one factor map transforming
      the dynamical system
      $(H_1 \times H_2, T_{H_1} \times T_{H_2}, \eta_{H_1} \times \eta_{H_2})$
      into the dynamical system $(\ttf^d,\eta)$. 
    \item The endomorphism $\Theta : x \mapsto ((qP_1)x,(qP_2)x)$
      from $\ttf^d$ to $H_1 \times H_2$ is a $q^d/m$-to one factor map
      transforming
      the dynamical system $(\ttf^d,\eta)$ into the dynamical system 
      $(H_1 \times H_2, T_{H_1} \times T_{H_2}, \eta_{H_1} \times \eta_{H_2})$.  
  \end{enumerate}
\end{prop}

\begin{proof}
  The inclusions $q\zzf^d \subset L \subset \zzf^d$ yields
  $[\zzf^d : L] \times [L : q\zzf^d] = [\zzf^d : q\zzf^d] = q^d$.
  Item~1 follows. 

  The map $\Sigma$ is a continuous surjective morphism of compact groups,
  so it transforms the Haar measure $\eta_{H_1} \times \eta_{H_2}$
  into $\eta$. Moreover, for every $(x_1,x_2) \in K_1 \times K_2$,
  \begin{eqnarray*}
    (p_{\ttf^d}(x_1),p_{\ttf^d}(x_2)) \in \Ker~\Sigma
    &\iff& x_1+x_2 \in \zzf^d \\
    &\iff& (x_1,x_2) \in R \text{ where }
           R := \{(P_1x,P_2x) : x \in \zzf^d\}.
  \end{eqnarray*}
  Thus $\Ker~\Sigma$ is the image of $R$ by the group morphism
  $(x_1,x_2) \mapsto (p_{\ttf^d}(x_1),p_{\ttf^d}(x_2))$
  from $(\rrf^d)^2$ to $(\ttf^d)^2$. Hence
  $$|\Ker~\Sigma| = [R : (K_1 \cap \zzf^d) \times (K_2 \cap \zzf^d)] 
  = [\zzf^d : L] = m,$$
  since the group isomorphism $x \mapsto (P_1x,P_2x)$ from 
  $\zzf^d$ to $R$ induces a group isomorphism from $L$ to
  $(K_1 \cap \zzf^d) \times (K_2 \cap \zzf^d)$.
  Item~2 follows.

  The map $\Theta$ is an epimorphism of compact groups, so it transforms
  $\eta$ into the Haar measure $\eta_{H_1} \times \eta_{H_2}$.
  For every $x \in \rrf^d$,
  \begin{eqnarray*}
    p_{\ttf^d}(x) \in \Ker~\Theta
    \iff (qP_1x,qP_2x) \in \zzf^d \times \zzf^d 
    \iff qx \in L.
  \end{eqnarray*}
  Thus $|\Ker~\Theta| = |p_{\ttf^d}(q^{-1}L)| =
  [q^{-1}L : \zzf^d] = [L : q\zzf^d]$. Item~3 follows.  
\end{proof}

\begin{coro}\label{condition for direct sum}
  Keep the notations of Proposition~\ref{finite-to-one factor maps}.
  The following statements are equivalent.
  \begin{enumerate}
    \item The matrices $P_1$ and $P_2$ have integral entries. 
    \item The subgroup $(K_1 \cap \zzf^d) \oplus (K_2 \cap \zzf^d)$ is
      $\zzf^d$. 
    \item The sum $H_1 + H_2 = \ttf^d$ is direct.  
    \item The endomorphism $\Sigma$ is an automorphism.   
  \end{enumerate}
\end{coro}

\begin{proof}
  Given $x \in \zzf^d$, $(P_1x,P_2x)$ is the unique
  $(x_1,x_2) \in K_1 \times K_2$ such that $x_1+x_2=x$, so  
  $x$ belongs to $(K_1 \cap \zzf^d) \oplus (K_2 \cap \zzf^d)$
  if and only if $P_1x$ and $P_2x$ belong to $\zzf^d$.  
  The equivalence $(1) \iff (2)$ follows.

  Proposition~\ref{finite-to-one factor maps} yields the equivalences
  $(2) \iff \big( m=1 \big) \iff (3) \iff (4)$. 
\end{proof}  

Since the unimodular polynomial $f_1 = \chi_{u_1}$ annihilates $u_1$,
the endomorphism $u_1^{-1}$ can be written as a linear combination
with integral coefficients of $\id, u_1, \ldots, u_1^{d_1-1}$, where
$d_1 = \dim H_1$. Therefore, the endomorphism $T_{H_1}$ is invertible,
whereas we will show that in the next subsection
(proposition~\ref{more on Gamma and H}) that the endomorphism
$T_{H_2}$ is exact.

Hence, the equivalent conditions 1, 2, 3 of
Corollary~\ref{condition for direct sum} are sufficient for the existence of 
an algebraic isomorphism (i.e. an isomorphism of compact groups) between $T$
and a direct product of some toral automorphism and some exact toral
endomorphism. The following result shows that they are also necessary.   

\begin{prop}~\label{condition for the existence of an algebraic isomorphism}
  If there exist a toral automorphism $T_1$ of $\ttf^{d_1}$,
  an exact toral endomorphism $T_2$ of $\ttf^{d_2}$, and a continuous
  group isomorphism $\Phi : \ttf^{d_1} \times \ttf^{d_2} \to \ttf^{d}$
  transforming the direct product $T_1 \times T_2$ into $T$,
  then the equivalent conditions 1, 2, 3, 4 of
  Corollary~\ref{condition for direct sum} hold.
\end{prop}

\begin{proof}
  The algebraic isomorphism $\Phi$ induces algebraic isomorphisms from 
  $\ttf^{d_1} \times \{0\}$ and $\{0\} \times \ttf^{d_2}$ to their images
  $G_1$ and $G_2$, and we have $G_1 \oplus G_2 = \ttf^d$.

  For each $i \in \{1,2\}$, the set $G_i$ is a connected closed subgroup
  of $\ttf^{d}$.
  By proposition~\ref{subtorus of a torus}, $G_i = p_{\ttf^{d}}(F_i)$ for
  some vector subspace $F_i$ of $\rrf^{d}$, and $F_i \cap \zzf^{d}$ is a
  lattice of $F_i$. The isomorphism $\Phi$ transforms $T_i$ into an
  epimorphism $S_i$ of $G_i$, which can be
  obtained by quotienting an isomorphism $s_i$ of the vector space $F_i$.

  Since $G_1 \cap G_2 = \{0_{\ttf^d}\}$, we have $F_1 \cap F_2 \subset \zzf^d$,
  so $F_1 \cap F_2 = \{0_{\rrf^d}\}$. Since $S_1$ is invertible and
  $S_2$ is exact, $\chi_{s_1}$ is unimodular whereas $\chi_{s_2}$ has no
  unimodular factor.
  Furthermore, $\chi_{s_1}\chi_{s_2} = \chi_A = f_1f_2$.
  As a result, $\chi_{s_1}=f_1$ and $\chi_{s_2}=f_2$,
  so $F_i = \Ker f_i(A) = K_i$ and $G_i = H_i$
  for each $i \in \{1,2\}$. Hence condition 3 of
  Corollary~\ref{condition for direct sum} holds. 
\end{proof}

Condition~2 of Corollary~\ref{condition for direct sum} is very similar to 
the necessary and sufficient condition of Proposition~4 in~\cite{Krzyzewski},
namely
$(\Ker f_1(A^*) \cap \zzf^d) \oplus (\Ker f_2(A^*) \cap \zzf^d) = \zzf^d$ 
(where $A^*$ denotes the transpose of $A$), which is established with duality
arguments by Krzy\.{z}ewski. The equivalence of the two conditions is not
obvious since the matrix $A^*$ is not necessarily $\zzf$-similar to $A$.
A simple counterexample is
$$A = \left(\begin{array}{cc}
              1 & -5 \\
              3 & -1
             \end{array}\right).$$

\section{Annex}\label{Annex}

\subsection{Closed subgroups of a finite-dimensional real vector space}


Let us recall a few definitions. Fix a real vector space $V$ with
finite-dimension $d$.

The \emph{rank} of a subset of $V$ is the dimension
of the vector subspace spanned by it. A \emph{sublattice} of $V$ is a
subgroup of $V$ generated by some family of (at most $d$)
linearly independant vectors of $V$. A \emph{lattice} of $V$ is a subgroup
generated by $d$ linearly independent vectors of $V$, i.e. a subgroup
$\zzf v_1 \oplus \cdots \oplus \zzf v_d$, where $(v_1,\ldots,v_d)$ is some
basis of $V$. Such a family of generators is called a $\zzf$-\emph{basis}
of $\Lambda$.

Given a lattice $\Lambda$ of $V$, a \emph{fundamental} domain of
$V/\Lambda$ is a Borel subset $D$ of $V$ which contains exactly one element
in each coset of $V/\Lambda$. A simple example of fundamental domain of
$V/\Lambda$ is $[0,1[v_1 \oplus \cdots \oplus [0,1[v_d$ where
$(v_1,\ldots,v_d)$ is a $\zzf$-basis of $\Lambda$. 

A key result in the theory of lattices claims that every discrete
(or equivalently locally finite) subgroup of $V$ is a sublattice of $V$.
This follows immediately from Proposition~\ref{characterization of lattices}
below. 

\begin{prop}\label{characterization of lattices}
  Let $V$ be a finite-dimensional real vector space and $G$ be a discrete
  subgroup of $V$. If $G$ spans $V$, then $G$ is a lattice of $V$.
\end{prop}

\begin{proof}
  Let $\mathbf{g} = (g_1,\ldots,g_d)$ be a basis of $V$ consiting of
  elements of $G$. Then $\Lambda := \zzf g_1 + \cdots + \zzf g_d$ is a
  lattice of $V$ and a subgroup of $G$.
  Since $G$ is discrete, the semi-open parallelipiped
  $[0,1[ g_1 + \cdots + [0,1[ g_d$ 
  contains finitely many elements of $G$. Therefore, the quotient group
  $G/\Lambda$ is finite. By Lagrange theorem, the order of each element of
  $G/\Lambda$ is finite. Call $m$ the exponent of $G/\Lambda$, i.e. the
  least common multiple of the orders of all elements of $G/\Lambda$.
  For every $g \in G$, $m(g+\Lambda) = 0_{G/\Lambda}$, so $mg \in \Lambda$. 
  Hence $G \subset (1/m)\Lambda$.

  As a result, for every $(v_1,\ldots,v_d) \in G^d$,
  $\det_{\bf g}(v_1,\ldots,v_d) \in (1/m)^d \zzf$,
  so the infimum of $|\det_{\bf g}(v_1,\ldots,v_d)|$ over all $d$-uples
  $(v_1,\ldots,v_d) \in G^d$ which are basis of $V$ is achieved.
  Fix a basis $(v_1,\ldots,v_d) \in G^d$ which achieves this infimum.
  Then $\zzf v_1 + \cdots + \zzf v_d \subset G$ and $G$ intersects the
  parallelipiped
  $P := [0,1] v_1 + \cdots + [0,1] v_d$ only at its vertices:
  otherwise $G$ would contain some point $g = \xi_1v_1+\cdots+\xi_dv_d$
  with $x_i \in [0,1]$ for all $i \in \odc 1,d \fdc$ and
  $x_j \in ~]0,1[$ for some $j \in \odc 1,d \fdc$, and the equality
  $\det_{\bf g}(v_1,\ldots,v_{j-1},g,v_{j+1},\ldots,v_d)
  = \xi_j\det_{\bf g}(v_1,\ldots,v_d)$ would lead to a contradiction. 
  One deduces that $G = \zzf v_1 + \cdots + \zzf v_d$ is a lattice of $V$.
\end{proof}
In the opposite direction, we have the following result. 

\begin{lemm}\label{accumulation point}
  Let $G$ be a closed subgroup of a finite-dimensional real vector space $V$.
  If $G$ has an accumulation point,
  then $G$ contains some non-null vector subspace of $V$.   
\end{lemm}

\begin{proof}
  Fix a norm $|\cdot|$ on $V$. Let $g$ be an accumulation point
  and $(g_n)_{n \ge 1}$ be a sequence of elements of $G \setminus \{g\}$
  converging to $g$. Then the vectors $h_n := g_n-g$ are non-null
  elements of $G$ converging to $0_V$. By compactness of the unit sphere,
  and by extracting a subsequence, one may assume that
  $(h_n/|h_n|)$ converges to some unit vector $u$. Since $(|h_n|)$ converges
  to $0$, one gets that for every $\lambda \in \rrf$,
  $$\lambda u
  = \lim_{n \to \infty} \lfloor \lambda/|h_n| \rfloor h_n \in G,$$
  so $\rrf u \subset G$. 
\end{proof}

We will also use a theorem giving the general structure
of all closed subgroups of a real finite-dimensional vector space.

\begin{prop}\label{closed subgroups of a real finite-dimensional vector space}
  {\bf (Closed subgroups of a real finite-dimensional vector space).}
  Let $V$ be a finite-dimensional real vector space and $S$ be a
  closed subgroup of $V$. Then there exist two vector subspaces 
  $F_1$, $F_2$ of $V$ and a lattice $\Lambda_1$ of $F_1$
  such that the sum $F_1+F_2$ is direct and $S = \Lambda_1 \oplus F_2$.
\end{prop}

\begin{proof}
  Let
  $$L := \{x \in V : \rrf x \subset S\}
  = \bigcap_{\lambda \in \rrf^*} \lambda^{-1}S.$$
  Then $L$ is the largest vector space contained in $S$. 

  Endow $V$ with any scalar product. We claim that $S \setminus L$ cannot
  contain vector arbitrarily close to $0$. Otherwise, their would exist
  some sequence $(x_n)_{n \ge 1}$ of vectors in $S \setminus L$
  converging to $0_V$. Their orthogonal projections
  $(p_{L^\perp}(x_n))_{n \ge 1}$ onto $L^\perp$ would form a sequence of
  nonzero vectors of $S \cap L^\perp$ converging to $0_V$. 
  By lemma~\ref{accumulation point}, $L^\perp$ would contain some
  non trivial vector space, which would contradict the definition of $L$.  
  
  Fix a basis of $L$ and complete it into a basis of
  $\mathrm{Span}(S)$ by including vectors of $S$. The additional vectors
  span a vector subspace $M$ such that $S = L \oplus M$.  
  Since the subgroup $S$ contains $L$, we derive $S = L \oplus (S \cap M)$. 
  The group $S \cap M$ is a lattice of $M$ since it is a full-rank
  subgroup of $M$ in which $0$ is isolated.

  The conclusion follows by taking $F_2=L$, $F_1 = M$,
  $\Lambda_1 = S \cap M$ and $F_0$ any vector subspace such that
  $V = F_0 \oplus \mathrm{Span}(S)$. 
\end{proof}

Using
Proposition~\ref{closed subgroups of a real finite-dimensional vector space},
one can prove that connected closed subgroups of a
torus are subtori. The next statement is used only in the proof of
Proposition~\ref{condition for the existence of an algebraic isomorphism}.

\begin{prop}\label{subtorus of a torus} {\bf (Subtorus of a torus).}
  Let $V$ be a finite-dimensional real vector space, $\Lambda$ be 
  a lattice of $V$ and $G$ be a connected closed subgroup of the torus
  $V/\Lambda$.
  Call $p_{V/\Lambda}$ the canonical projection from $V$ onto $V/\Lambda$.
  Then there exists a vector subspace $F$ of $V$ such that
  $G = p_{V/\Lambda}(F)$ and $F \cap \Lambda$ is a lattice of $F$.
  Hence $G$ is canonically isomorphic to the torus $F/(F \cap \Lambda)$.
\end{prop}

\begin{proof}
  Since $p_{V/\Lambda}$ is a continuous group morphism, the set
  $S := p_{V/\Lambda}^{-1}(G)$ is a closed subgroup of $V$ containing $\Lambda$,
  so it has full rank. By
  proposition~\ref{closed subgroups of a real finite-dimensional vector space},
  there exists two vector subspaces $F_1$, $F_2$ of $V$ and a lattice
  $\Lambda_1$ of $F_1$ such that $V = F_1 \oplus F_2$ and
  $S = \Lambda_1 \oplus F_2$. By surjectivity of $p_{V/\Lambda}$,
  we derive $G = p_{V/\Lambda}(S) = p_{V/\Lambda}(\Lambda_1 \oplus F_2)$.

  Call $p_{F_1}$ the projection from $V$ to $F_1$ along $F_2$
  and $p_{V/(F_2+\Lambda)}$ the canonical projection from $V$ onto
  $V/(F_2+\Lambda)$. Since $p_{F_1}(\Lambda) \subset \Lambda + F_2$,  
  the continuous group morphism $p_{V/(F_2+\Lambda)} \circ p_{F_1}$
  vanishes on $\Lambda$. Thus there exists
  a continuous group morphism $f : V/\Lambda \to V/(F_2+\Lambda)$ such that
  $f \circ p_{V/\Lambda} = p_{V/(F_2+\Lambda)} \circ p_{F_1}$. Hence
  $$f(G) = f \circ p_{V/\Lambda} (S)
  = p_{V/(F_2+\Lambda)} \circ p_{F_1} (S) = p_{V/(F_2+\Lambda)}(\Lambda_1).$$
  Since $G$ is connected, the discrete group
  $p_{V/(F_2+\Lambda)}(\Lambda_1)$ is connected, hence trivial.
  As a result, $\Lambda_1 \subset F_2+\Lambda$, so
  $G = p_{V/\Lambda}(\Lambda_1 \oplus F_2) = p_{V/\Lambda}(F_2)$.
  The restriction of $p_{V/\Lambda}$ to $F_2$ is a continuous group morphism
  with kernel $F_2 \cap \Lambda$, so $G$ is canonically isomorphic to
  $F_2/(F_2 \cap \Lambda)$. The discrete subgroup $F_2 \cap \Lambda$
  is therefore a lattice of $F_2$, otherwise it would contradict the
  compactness of $G$.  
\end{proof}

\subsection{Structure of a particular subgroup} 

In this subsection, we fix a finite-dimensional real vector space $V$, a lattice 
$\Lambda$ of $V$ and an endomorphism $u \in GL(V)$ such that 
$u(\Lambda) \subset \Lambda$. We study the structure of the subgroup
$$\Gamma := \bigcup_{n \ge 0}u^{-n}(\Lambda).$$

\begin{lemm}\label{condition for Gamma to be a lattice}
  The subset $\Gamma$ is a subgroup of $V$, stable by $u$ and $u^{-1}$,
  and contained in the union of $|\!\det u|^{-n}\Lambda$ over all
  integers $n \ge 1$.  
  Moreover, $\det u$ is a non-null integer and 
  \begin{itemize}
  \item either $|\!\det u| = 1$ and $\Gamma$ is discrete, equal to $\Lambda$,    
  \item or $|\!\det u| \ge 2 $ and $\overline{\Gamma}$ contains some
    non-trivial vector subspace of $V$.  
  \end{itemize}
\end{lemm}

\begin{proof}
  Since $(u^{-n}(\Lambda))_{n \ge 0}$ is a non-decreasing sequence of
  subgroups and its union $\Gamma$ is still a subgroup of $V$, closed
  under $u$ and $u^{-1}$ . 
  
  Fix a $\zzf$-basis $\bc$ of $\Lambda$. Since $u(\Lambda) \subset \Lambda$,
  the matrix $A$ of $u$ in the basis $\bc$ has integral coefficients, so
  $\det u = \det A$ is an integer, non-null because $u \in GL(V)$. The
  adjugate matrix (i.e. the transpose of the comatrix) $\tilde{A}$ of $A$
  also has integral coefficients. For every $n \ge 0$,
  $A^{-n} = (\det A)^{-n} \tilde{A}^n$, so
  $u^{-n}(\Lambda) \subset |\!\det u|^{-n}\Lambda$. Therefore, 
  $$\Gamma \subset \bigcup_{n \ge 0} |\!\det u|^{-n}\Lambda.$$   

  If $|\!\det u| = 1$, then the matrix of $u^{-1}$ in the basis $\bc$ has
  also integral coefficients, so $u^{-1}(\Lambda) \subset \Lambda$. Thus
  the sequence $(u^{-n}(\Lambda))_{n \ge 0}$ is constant and
  $\Gamma = \Lambda$. 

  If $|\!\det u| \ge 2$, then some vector of the $\zzf$-basis $\bc$
  must be outside $u(\Lambda)$, because the determinant in $\bc$ of $d$
  arbitrary vectors of $u(\Lambda)$ is always a multiple of $\det u$.  
  Hence, the inclusion $u(\Lambda) \subset \Lambda$ is strict, so
  the sequence $(u^{-n}(\Lambda))_{n \ge 0}$ is strictly increasing.  
  One can choose for each $n \ge 0$ a vector
  $x_n \in u^{-(n+1)}(\Lambda) \setminus u^{-n}(\Lambda)$. Note that all
  the vectors $(x_n)_{n \ge 0}$ are necessarily different. By subtracting
  an appropriate sequence of vectors of $\Lambda$, one may additionally
  assume that $(x_n)_{n \ge 0}$ remains in some bounded fundamental
  domain of $V/\Lambda$. Thus, the sequence $(x_n)_{n \ge 0}$ has an
  limit point and $\Gamma$ has an accumulation point. Then
  Lemma~\ref{accumulation point} shows that $\overline{\Gamma}$ contains
  some non-trivial vector subspace of $V$.

  The proof is complete.
\end{proof}

In the opposite direction, we have the result below. 

\begin{lemm}\label{sufficient condition for Gamma to be dense}
  If the characteristic polynomial $\chi_u$ has no non-trivial unimodular
  (i.e. monic, with integral coefficients and constant term $\pm 1$)
  factor, then $\Gamma$ is dense in $V$.    
\end{lemm}

\begin{proof}
  The closed subgroup $\overline{\Gamma}$ has full rank since it contains
  $\Lambda$. By
  Proposition~\ref{closed subgroups of a real finite-dimensional vector space}, 
  there exist two vector subspaces $F_1$, $F_2$ of $V$ and a lattice
  $\Lambda_1$ of $F_1$ such that the sum $F_1+F_2$ is direct and
  $\overline{\Gamma} = \Lambda_1 \oplus F_2$. Then $F_2$
  is the largest vector subspace contained in $\overline{\Gamma}$.
  Since $u(\overline{\Gamma}) = \overline{u(\Gamma)} = \overline{\Gamma}$,
  we have $u(F_2) = F_2$.

  Taking quotients, $u$ yields an endomorphism $\overline{u}$ of $V/F_2$.
  Since the canonical projection $p_{V/F_2}$ induces an isomorphism from $F_1$
  to $V/F_2$, the subset $p_{V/F_2}(\Lambda_1)$ is a lattice of $V/F_2$.
  The equality
  $$u(\Lambda_1) \oplus F_2 = u(\Lambda_1) \oplus u(F_2) = u(\overline{\Gamma})
  = \overline{\Gamma} = \Lambda_1 \oplus F_2,$$
  yields $\overline{u}(p_{V/F_2}(\Lambda_1)) = p_{V/F_2}(\Lambda_1)$. 
  Thus, the matrix of $\overline{u}$ in a $\zzf$-basis of
  $p_{V/F_2}(\Lambda_1)$ has integral coefficients and its determinant is
  $\pm 1$, so the characteristic polynomial $\chi_{\overline{u}}$
  is unimodular. Since $\chi_{\overline{u}}$ divides $\chi_u$,
  the assumption yields $\chi_{\overline{u}}=1$, so the quotient space
  $V/F_2$ is trivial and $\overline{\Gamma} = V$.   
\end{proof}

Actually, the converse of
Lemma~\ref{sufficient condition for Gamma to be dense} is true, and this
will follow from the next result which gives a complete description of
the closure of $\Gamma$. 

\begin{prop}\label{structure of Gamma}
  Write $\chi_u = f_1f_2$, where $f_1$ is the product of all unimodular
  irreducible in $\zzf[X]$ (or equivalently in $\qqf[X]$) monic
  factors of $\chi_u$ and $f_2$ the product of all non-unimodular
  irreducible monic factors of $\chi_A$, counted with their multiplicities.
  Let $K_1 := \Ker f_1(u)$ and $K_2 := \Ker f_2(u)$. Then
  \begin{enumerate}
  \item One has $V = K_1 \oplus K_2$. The associated projections
    $p_1$ and $p_2$ onto $K_1$ and $K_2$ are polynomials of $u$
    with rational coefficients. 
  \item The vector subspaces $K_1$ and $K_2$ are stable under $u$.
    The induced endomorphisms $u_1$ and $u_2$
    of $K_1$ and $K_2$ satisfy $\chi_{u_1} = f_1$, $\chi_{u_2} = f_2$
    and $\det u_1 = \pm 1$. 
  \item One has $\overline{\Gamma} = p_1(\Lambda) \oplus K_2$ and
    $p_1(\overline{\Gamma}) = p_1(\Gamma) = p_1(\Lambda)$.
    Moreover, the subgroup $p_1(\Lambda)$ is a lattice of $K_1$. 
  \item Call $p_{V/\Lambda}$ the canonical projection from $V$ 
    onto $V/\Lambda$. Then $p_{V/\Lambda}(\overline{\Gamma}) 
    = p_{V/\Lambda}(K_2)$.
  \end{enumerate}
\end{prop}

\begin{proof}
  Since $f_1f_2$ annihilates $u$ and since $f_1$ and $f_2$ are
  relatively prime in $\qqf[X]$, one has
  $V = \Ker (f_1f_2)(u) = K_1 \oplus K_2$. Moreover, 
  there exist $g_1$ and $g_2$ in $\qqf[X]$ such that $g_1f_1+g_2f_2 = 1$
  and the endomorphisms $p_1 := g_2f_2(u)$ and $p_2:= g_1f_1(u)$ are 
  the associated projections onto $K_1$ and $K_2$. Item 1 follows. 

  The vector subspaces $K_1$ and $K_2$ are stable by $u$ and $u^{-1}$.
  Calling $u_1$ and $u_2$ the endomorphisms induced
  by $u$ on $K_1$ and $K_2$, we get by construction that
  the endomorphisms $f_1(u_1)$ and $f_2(u_2)$ are null. Since the 
  characteristic polynomial of $u_1$ (respectively $u_2$) has the same
  irreducible factors as the minimal polynomial of $u_1$
  (respectively $u_2$), one derives that $\chi_{u_1}$ divides some power
  of $f_1 $ and $\chi_{u_2}$ divides some power of $f_2$. Furthermore,
  $\chi_{u_1}\chi_{u_2} = \chi_u = f_1f_2$ and $f_1,f_2$ are relatively prime.
  As a result, $\chi_{u_1} = f_1$ and $\chi_{u_2} = f_2$.
  Since $f_1$ is unimodular, we get $\det u_1 = \pm 1$. Item 2 follows. 

  Fix an integer $q \ge 1$ such that $qg_if_i$ has integral coefficients
  for all $i \in \{1,2\}$. Since $u(\Lambda) \subset \Lambda$,
  $qp_1(\Lambda) = qg_if_i(\Lambda) \subset \Lambda$, so
  $p_i(\Lambda) \subset (1/q)\Lambda$. Using that $\Lambda$ is a lattice
  of $V$, we derive that $p_i(\Lambda)$ is a lattice of $K_i$.
  Since the endomorphims $p_i$ and $u$ commute, we have also
  $$u_i(p_i(\Lambda)) = u(p_i(\Lambda)) = p_i(u(\Lambda)) \subset
  p_i(\Lambda)$$
  and
  $$p_i(\Gamma) = \bigcup_{n \ge 0} p_i(u^{-n}(\Lambda))
  = \bigcup_{n \ge 0} u^{-n}(p_i(\Lambda))
  = \bigcup_{n \ge 0} u_i^{-n}(p_i(\Lambda)).$$ 

  On the one hand, $\det u_1 = \pm 1$, so
  Lemma~\ref{condition for Gamma to be a lattice}
  yields $p_1(\Gamma) = p_1(\Lambda)$. 
  By continuity of $p_1$ and since
  $p_1(\Lambda)$ is closed, we derive 
  $p_1(\overline{\Gamma}) \subset \overline{p_1(\Gamma)} = p_1(\Lambda)$, so
  $$\overline{\Gamma}
  \subset p_1(\overline{\Gamma}) + p_2(\overline{\Gamma})
  \subset p_1(\Lambda) + K_2.$$
  On the other hand, the only unimodular factor of $\chi_{u_2}$ is $1$, so
  Lemma~\ref{sufficient condition for Gamma to be dense} ensures that
  $p_2(\Gamma)$ is dense in $K_2$. Since
  $u(\overline{\Gamma}) \subset \overline{u(\Gamma)} \subset \overline{\Gamma}$
  and since $qp_2 \in \zzf[u]$, we derive, by continuity of $p_2$,   
  $$K_2 = qK_2 = q\overline{p_2(\Gamma}) \subset
  qp_2(\overline{\Gamma}) \subset \overline{\Gamma}$$ whereas 
  $$p_1(\Lambda) \subset ({\id}_V-p_2)(\Gamma) \subset \Gamma+K_2.$$
  Hence, the reverse inclusion 
  $$p_1(\Lambda) + K_2 \subset \Gamma + K_2 \subset \overline{\Gamma}$$
  and item 3 follow. 

  
  The inclusion $K_2 \subset \overline{\Gamma}$ yields 
  $p_{V/\Lambda}(K_2) \subset p_{V/\Lambda}(\overline{\Gamma})$. 
  Conversely, item 3 yields
  $$p_{V/\Lambda}(\overline{\Gamma}) \subset p_{V/\Lambda}(p_1(\Lambda))
    + p_{V/\Lambda}(K_2).$$ 
  Furthermore, $p_1(\Lambda) = (\id_V-p_2)(\Lambda) \subset \Lambda + K_2$, 
  so $p_{V/\Lambda}(p_1(\Lambda)) \subset p_{V/\Lambda}(K_2)$.
  The reverse inclusion and item 4 follows.
\end{proof}

\subsection{Haar measures on Hausdorff compact groups}

Given a Hausdorff compact group $G$, the Haar measure $\eta_G$
is the unique probability measure on $(G,\bc(G))$ which is invariant
by left translations (equivalently, by right translations).
The existence and the uniqueness follow from Haar's theorem. 

\begin{lemm}\label{image of Haar measure by a surjective morphism}
  If $f$ is a surjective continuous endomorphism from a compact group $G$
  to a compact group $G'$, then $f(\eta_G) = \eta_{G'}$.
\end{lemm}

\begin{proof}
  For each $a \in G$, the translation $\tau_a : x \mapsto ax$ of the group
  $G$ preserves $\eta_G$. Furthermore, $f \circ \tau_a = \tau'_{f(a)} \circ f$,
  where $\tau'_{f(a)}$ denotes left translation by $f(a)$ in the group $G'$.
  Thus $\tau'_{f(a)}(f(\eta_G)) = f(\tau_a(\eta_G)) = f(\eta_G)$.
  Since $f$ is surjective, the probability measure $f(\eta_G)$ is invariant
  by left-translations, so it is $\eta_{G'}$.
\end{proof}

\begin{lemm}\label{quotient by a closed normal subgroup}
  Let $G$ be a compact group and $H$ be a closed normal subgroup of $G$.
  Assume that the topology of $G$ can be derived from some distance
  which is invariant by translation. Then 
  \begin{enumerate}
  \item the quotient $G/H$ is a metrizable compact group;
  \item for every coset $C=aH$, the measure $\delta_a*\eta_H$
    does not depend on $a \in C$ and can be denoted by $\eta_C$
    (uniform distribution on $C$);
  \item for every coset $C=aH$, for every probability measure $\pi$ on $C$,
    $\pi * \eta_H = \eta_C$. 
  \item one has  
  $$\eta = \int_{G/H} \eta_C ~ \d\eta_{G/H}(C).$$
  \end{enumerate}
\end{lemm}

The last equality follows Weil's formula (see for example~\cite{Iozzi},
Theorem 2.4).

\begin{proof}
  Let $d$ be a distance on $G$ which is invariant by translations and provides
  the topology $G$. For every $x \in G$, the lower bound  
  $\dist(x,H) = \inf\{d(x,h) : h \in H\}$ is achieved at some $h_x \in H$, so
  $$\dist(x^{-1},H) \le d(x^{-1},h_x^{-1}) = d(1_G,xh_x^{-1}) = d(h_x,x)
  = \dist(x,H).$$
  By symmetry, we derive $\dist(x^{-1},H) = \dist(x,H)$.  

  One checks that the formula
  $d(aH,bH) := \dist(a^{-1}b,H)$ is coherent and defines a metric on $G/H$,
  which is also invariant by translations. 
  Call $p_{G/H}$ the canonical projection $p_{G/H}$ from $G$ onto $G/H$.
  Then for every $r>0$, $p_{G/H}(B(1_G,r)) = B(1_{G/H},r)$, where the letter
  $B$ stands for open ball. Indeed, for every $x \in G/H$,
  $$x \in B(1_G,r) \implies
  d(1_{G/H},xH) = d(1_GH,xH) = \dist(x,H) \le d(x,1_G) < r;$$
  conversely, if $d(1_GH,xH)  = \dist(x,H) < r$, then $d(x,h_x)<r$ for
  some $h_x \in H$, so $d(xh_x^{-1},1_G)<r$ and
  $xH = xh_x^{-1}H \in p_{G/H}(B(1_G,r))$.

  Now, let $V \subset G/H$. Then
  \begin{eqnarray*}
    V \text{ is a neighourhood of } 1_{G/H}=H
    &\iff& p_{G/H}^{-1}(V) \text{ is a neighourhood of } 1_G \\
    &\iff& p_{G/H}^{-1}(V) \supset B(1_G,r) \text{ for some } r>0 \\
    &\iff& V \supset p_{G/H}(B(1_G,r)) \text{ for some } r>0 \\
    &\iff& V \supset B(1_{G/H},r) \text{ for some } r>0.
    \end{eqnarray*}
  As a result, the quotient topology on $G/H$ coincides with the
  topology given by $d$. Item~1 follows. 
  
  Let $a \in G$. The measure $\delta_a*\eta_H$ is the image of $\eta_H$ by
  the translation $\tau_a : x \mapsto ax$. If $bH=aH$ with $b \in G$,
  then $b=ah$ for some $h \in H$, and $\tau_b(\eta_H)
  = \tau_a(\tau_h(\eta_H)) = \tau_a(\eta_H)$. Item 2 follows.

  Let $\pi$ be any probability measure on the coset $C=aH$.
  Then $\nu := \delta_{a^{-1}} * \pi$ is a probability measure on $H$
  and $\pi * \eta_H = \delta_a * \nu * \eta_H = \delta_a*\eta_H = \eta_C$.
  Item 3 follows.

  Since $p_{G/H}$ is a surjective continuous group morphism from $G$
  onto $G/H$, one has $p_{G/H}(\eta_G) = \eta_{G/H}$. Hence,
  the disintegration of $\pi$ by $p_{G/H}$ is given by
  $$\pi = \int_{G/H} \pi_C ~ \d\eta_{G/H}(C),$$
  where $\pi_C$ is some probability measure carried by $C$, depending
  measurably on $C$. Hence
  $$\eta_G = \eta_G * \eta_H = \int_{G/H} (\pi_C * \eta_H)~ \d\eta_{G/H}(C)
  = \int_{G/H} \eta_C~ \d\eta_{G/H}(C),$$
  Item~4 follows. 
\end{proof}
\subsection*{Acknowledgements}

I would like to thank Jean Brossard, who has significantly simplified
some of the proofs in the present article.

\end{document}